\documentclass[reqno,11pt]{amsart}

\usepackage{amsmath, amsthm, amssymb, amsfonts, amsthm}
\usepackage{amsxtra, amscd, geometry, graphicx,epic,eepic}
\usepackage[all,cmtip]{xy}
\usepackage{mathrsfs}
\usepackage{hyperref}
\usepackage{color}

\usepackage{holtpolt}

\newtheorem{proposition}{Proposition}
\newtheorem{theorem}[proposition]{Theorem}         
\newtheorem{cor}[proposition]{Corollary}
\newtheorem{lemma}[proposition]{Lemma}

\theoremstyle{definition}
\newtheorem{remark}[proposition]{Remark}

\newtheorem*{rmk*}{Remark}

\renewcommand{\tilde}[1]{\widetilde{#1}}%

\newcommand{\N}{\mathbb N}
\newcommand{\Q}{\mathbb Q}
\newcommand{\Z}{\mathbb Z}

\newcommand{\R}{\mathbb R}

\newcommand{\II}{\mathbb I}
\newcommand{\BB}{\mathcal B}
\newcommand{\GG}{\mathcal G}
\newcommand{\UU}{\mathcal U}

\title[$\alpha$-expansions with odd partial quotients]{$\alpha$-expansions with odd partial quotients}

\author{Florin P. Boca}
\author{Claire Merriman}

\address{Department of Mathematics, University of Illinois at Urbana-Champaign,
Urbana, IL 61801, U.S.A.}

\address{E-mail: fboca@math.uiuc.edu;  emerrim2@illinois.edu}

\begin{document}

\date{\today}

\begin{abstract}
We consider an analogue of Nakada's $\alpha$-continued fraction transformation
in the setting of continued fractions with odd partial quotients.
More precisely, given $\alpha \in [\frac{1}{2}(\sqrt{5}-1),\frac{1}{2}(\sqrt{5}+1)]$,
we show that every irrational number $x\in I_\alpha=[\alpha-2,\alpha)$ can be uniquely represented as
$$
x= \polter{e_1 (x;\alpha)}{d_1 (x;\alpha)} +\polter{e_2 (x;\alpha)}{d_2 (x;\alpha)}+\polter{e_3 (x;\alpha)}{d_3 (x;\alpha)}+\cdots ,
$$
with $e_i(x;\alpha) \in \{ \pm 1\}$ and $d_i(x;\alpha) \in 2\N -1$
determined by the iterates of the transformation
$$\varphi_\alpha (x) := \frac{1}{| x|} - 2 \bigg[ \frac{1}{2| x|} +\frac{1-\alpha}{2} \bigg]-1$$
of $I_\alpha$. We also describe the natural extension of $\varphi_\alpha$ and prove that the endomorphism
$\varphi_\alpha$ is exact.
\end{abstract}

\maketitle

\section{Introduction}
Seminal work of Nakada \cite{Nak} investigated in depth the
$\alpha$-expansion of irrational numbers in the interval $[\alpha-1,\alpha)$,
associated with his Gauss map defined by
\begin{equation*}
f_\alpha (x)=\frac{1}{| x|} -\left[ \frac{1}{| x|} +1-\alpha \right],
\qquad  x\neq 0,
\end{equation*}
for $\alpha$ in the range $[1/2,1]$.
Every irrational number $x\in [\alpha-1,\alpha)$ has a unique expansion
\begin{equation*}
x= \polter{e_1(x;\alpha)}{d_1(x;\alpha)} + \polter{e_2(x;\alpha)}{d_2(x;\alpha)} +
\polter{e_3(x;\alpha)}{d_3 (x;\alpha)} + \cdots
\end{equation*}
with $e_i (x;\alpha)\in \{ \pm 1\}$ and $d_i (x;\alpha) \in \N$, where additional constraints
on pairs of consecutive digits may occur.
This interpolates between the regular continued fractions (RCF) when $\alpha=1$ and
the nearest integer continued fractions (NICF) when $\alpha=\frac{1}{2}$.
Kraaikamp \cite{Kr} investigated additional properties of this
$\alpha$-expansion in the range $\alpha \in [\frac{1}{2},1]$,
and several other works analyzed the more challenging
situation where $\alpha < \frac{1}{2}$ (see, e.g., \cite{AS,CT,JK,KSS,LM,MCM,NN,Na,Ti}).

Nakada's $\alpha$-expansions became part of a broader class of
continued fractions. They correspond to $q=3$ in the family of $\alpha$-Rosen continued fractions
investigated in detail by Dajani, Kraaikamp, and Steiner \cite{DKS}. These continued fractions
are generated
by the Gauss map $T_{\alpha;q} :[\lambda_q (\alpha -1),\lambda_q \alpha) \rightarrow[\lambda_q (\alpha-1),\lambda_q \alpha)$
given by
\begin{equation}\label{1.1}
T_{\alpha;q} (x):=\frac{1}{| x|} -\lambda_q \left[
\frac{1}{\lambda_q | x|} +1-\alpha \right], \qquad x\neq 0,
\end{equation}
where $\lambda_q =2\cos (\frac{\pi}{q})$ and $\alpha \in [\frac{1}{2},\frac{1}{\lambda_q}]$.

In a similar spirit, this paper considers a new class of continued fraction transformations
$\varphi_\alpha$ on $I_\alpha:=[\alpha-2,\alpha)$, with $\alpha \in [g,G]$,
$g=\frac{1}{2}(\sqrt{5}-1)$, $G=\frac{1}{2} (\sqrt{5}+1)$, defined by $\varphi_\alpha (0)=0$ and
\begin{equation}\label{1.2}
\varphi_\alpha (x) := \frac{1}{| x|} - d_\alpha (x),\qquad
d_\alpha (x) := 2 \left[ \frac{1}{2| x|} +\frac{1-\alpha}{2} \right]+1 ,
\qquad x\in I_\alpha \setminus \{ 0\}.
\end{equation}
The map $\varphi_\alpha$ coincides with the Gauss extended odd continued fraction (OCF) map when $\alpha =1$
(\cite{Rie,Sch0}, see also \cite[Section~3.1]{BL}),
and with the Gauss grotesque continued fraction (GCF) map at $\alpha=G$ \cite{Rie,Seb}.
Although formula \eqref{1.2} for the $\alpha$-OCF Gauss map looks like the limiting case $q\rightarrow\infty$ in \eqref{1.1},
the two types of continued fractions are quite different. In particular,
when $q\rightarrow \infty$, the interval $[\frac{1}{2},\frac{1}{\lambda_q}]$ shrinks to
the singleton set $\{ \frac{1}{2}\}$, whereas the range of $\alpha$ in $\varphi_\alpha$ may extend
beyond $[g,G]$.

We first analyze the map $\Phi_\alpha$ defined by
\begin{equation}\label{1.3}
\Phi_\alpha (x,y) := \left( \varphi_\alpha (x), \frac{1}{d_\alpha (x)+ e(x) y}\right),
\end{equation}
where $d_\alpha (x)$ is as in \eqref{1.2} and $e(x):=\operatorname{sign} (x)$.
In Section \ref{alpha_gauss}, we define a region $\Omega_\alpha \subseteq [\alpha-2,\alpha) \times [0,G)$
on which $\Phi_\alpha$ acts bijectively modulo a set of Lebesgue measure zero,
with finite invariant measure $d\mu_\alpha =(1+xy)^{-2} dx dy$.
Given the concrete form of $\Phi_\alpha$, we show in Section \ref{alpha_ergodic} that $(\Omega_\alpha,\BB_{\Omega_\alpha},(3\log G)^{-1} \mu_\alpha,\Phi_\alpha)$ gives the minimal invertible extension
of $(I_\alpha,\BB_{I_\alpha},\nu_\alpha,\varphi_\alpha)$, also called the \emph{natural extension}. Here $\nu_\alpha$ is a $\varphi_\alpha$-invariant probability Lebesgue absolutely continuous measure, explicitly computed in Corollary \ref{Cor12} and $\BB_{C}$
denotes the Borel $\sigma$-algebra on $C$.

Using the explicit description of $\Omega_\alpha$, we also prove in Section \ref{alpha_ocf}
that every number $x\in \II_\alpha:=I_\alpha \setminus \Q$ has a unique representation as
\begin{equation}\label{1.4}
x= \polter{e_1}{d_1} +\polter{e_2}{d_2}+\cdots =\cfrac{e_1}{d_1+\cfrac{e_2}{d_2+\dots}},
\end{equation}
where $e_i =e_i(x;\alpha)=e(\varphi_\alpha^{i-1} (x)) \in \{ \pm 1\}$ and
$d_i =d_i(x;\alpha)=d_\alpha (\varphi_\alpha^{i-1}(x)) \in 2\N -1$.
When $\alpha=g$, we require that $(d_i,e_i)\neq (1,1)$.
When $\alpha=1$, we require
that $(d_i,e_{i+1}) \neq (1,-1)$ as for OCF expansions, while when
$\alpha=G$ we require $(d_i,e_i)\neq (1,-1)$.
In general, our expansion \eqref{1.4} will also impose certain restrictions on
consecutive digits, among them:
\begin{itemize}
\item[(a)]
when $g<\alpha <1$, $(d_i,e_i)=(1,1)$ implies $e_{i+1}=1$, and
$(d_i,e_i)=(1,-1)$ implies that $(d_{i+1},e_{i+1})\neq (1,-1),(3,-1)$;
\item[(b)] when $1<\alpha<G$,
$(d_i,e_i) =(1,-1)$ implies $e_{i+1}=1$, and $(d_i,e_i)=(1,1)$ implies
$(d_{i+1},e_{i+1}) \neq (1,-1)$.
\end{itemize}

In the situation of Nakada's $\alpha$-RCF expansions, the sequence of denominators
$q_n(x;\alpha)$ of an irrational number $x$ is monotonically increasing.
For our $\alpha$-OCF expansions, this is no longer the case, and the change in sign
of $q_{n+1}(x;\alpha)-q_n (x;\alpha)$ is pretty subtle. Nevertheless, a thorough analysis
of the ratio of consecutive denominators, involving four consecutive iterates of the
natural extension $\Phi_\alpha$, enables us to prove the estimate $q_n(x;\alpha) \geq q B^n$
for every $x\in {\mathbb I}_\alpha$,
with $B=(5G-2)^{1/5} \approx 1.43524 > \sqrt{2}$ and $q\approx 0.03438$. Due to a standard argument (see e.g. \cite{Nak}),
this will suffice to establish ergodicity and exactness of $\varphi_\alpha$ for
every $\alpha \in [g,G]$ (cf. Theorem \ref{Thm16} below).
We also show in Theorem \ref{Thm20} that the corresponding Kolmogorov-Sinai entropy is equal to $\frac{\pi^2}{9\log G}$ for
every $\alpha \in [g,G]$.

Interestingly, the shape of the natural extension domain for the
$\alpha$-RCF transformation with $\alpha \in [\sqrt{2}-1,\frac{1}{2}]$ investigated in \cite{KSS,LM} is
similar to the domain of $\Phi_\alpha$ for $\alpha \in [g,1]$,
and the domain for $\alpha$-RCF transformations with $\alpha \in [\frac{1}{2},g]$ from \cite{Nak} is similar
to the domain of $\Phi_\alpha$ for $\alpha \in [1,G]$.

In this paper the Lebesgue measure on $\R$ will be denoted by $\lambda$.
A subset of $\R$ of Lebesgue measure zero will be called a \emph{null-set}.

\section{A skew-shift over the $\alpha$-OCF Gauss map}\label{alpha_gauss}
In this section we display a region $\Omega_\alpha \subseteq [\alpha-2,\alpha) \times [0,G)$,
which is $\Phi_\alpha$-invariant and such that $\Phi_\alpha$ is invertible, bi-measurable, and non-singular
on $\Omega_\alpha$ up to a null-set (see Fig. \ref{Figure1}). In Section \ref{alpha_ergodic} we show that this
gives the natural extension of $\varphi_\alpha$ for some appropriate invariant measure.

We consider the rank-one cylinders
\begin{equation*}
\langle b \rangle_\alpha =\{ x\in I_\alpha: d_\alpha (x)=| b| ,e(x) =\operatorname{sign} (b)\},
\quad b\in \Z^* =\Z \setminus \{ 0\},
\end{equation*}
that is $\langle b\rangle_k =\emptyset$ if $b$ is even and $\langle -1\rangle_\alpha =[ \alpha-2,-\frac{1}{1+\alpha})$,
$\langle 1\rangle_\alpha =( \frac{1}{1+\alpha},\alpha )$, and
\begin{equation*}
\begin{split}
&
\langle -2k-1 \rangle_\alpha = \left[ -\frac{1}{2k-1+\alpha},-\frac{1}{2k+1+\alpha}\right), \\
&
\langle 2k+1\rangle_\alpha = \left( \frac{1}{2k+1+\alpha}, \frac{1}{2k-1+\alpha} \right]
\qquad \mbox{\rm if $k\geq 1$.}
\end{split}
\end{equation*}
Note that $\langle 1\rangle_g =\emptyset$, $\langle -1\rangle_G =\emptyset$, and
$\langle b\rangle_\alpha \neq \emptyset$ for every odd positive integer $b$ when $\alpha \in (g,G)$.

With $d_\alpha (\alpha-2)$, $\varphi_\alpha (\alpha -2)$ as in formula \eqref{1.2} and
\begin{equation*}
d_\alpha (\alpha) :=2\bigg[ \frac{1}{2\alpha} +\frac{1-\alpha}{2} \bigg]+1, \qquad
\varphi_\alpha (\alpha) := \frac{1}{\alpha} -d_\alpha (\alpha) ,
\end{equation*}
it is elementary to check the following statement.

\begin{lemma}\label{Lemma1}
{\em (i)} When $g< \alpha \leq G$ we have $d_\alpha (\alpha)=1$, and
when $g\leq \alpha< G$ we have $d_\alpha (\alpha-2)=1$.

{\em (ii)} When $g< \alpha \leq 1$ we have
\begin{equation}\label{2.1}
\begin{split}
& -\frac{1}{3+\alpha} < \varphi_\alpha (\alpha-2) \leq 0 \leq \varphi_\alpha (\alpha) < \frac{1}{1+\alpha},
 \\
& d_\alpha ( \varphi_\alpha (\alpha)) =d_\alpha ( -\varphi_\alpha (\alpha-2)) -2.
\end{split}
\end{equation}

{\em (iii)} When $1\leq \alpha < G$ we have
\begin{equation}\label{2.2}
\begin{split}
& -\frac{1}{1+\alpha} < \varphi_\alpha (\alpha) \leq 0 \leq \varphi_\alpha (\alpha -2) < \alpha,
 \\
& d_\alpha ( \varphi_\alpha (\alpha -2)) +2 = d_\alpha ( -\varphi_\alpha (\alpha)).
\end{split}
\end{equation}

{\em (iv)} When $\alpha \in(g,1)\cup(1,G)$ we have
\[\frac{1}{\varphi_\alpha(\alpha)}+\frac{1}{\varphi_\alpha (\alpha-2)} =-2 \qquad \mbox{and}\qquad
\varphi_\alpha^2 (\alpha)=\varphi_\alpha^2 (\alpha-2).\]
When $\alpha=1$, $\varphi_\alpha^2 (\alpha)=\varphi_\alpha^2 (\alpha-2)$ holds.
\end{lemma}

This provides for $g < \alpha < 1$ the identities
$$\alpha= \polter{1}{1} +\polter{1}{d_\alpha (\varphi_\alpha (\alpha))+\varphi^2_\alpha(\alpha)}
\qquad \mbox{\rm and} \qquad
\alpha-2= \polter{-1}{1} +\polter{-1}{d_\alpha (\varphi_\alpha(\alpha-2))+\varphi^2_\alpha(\alpha)}.$$
For $1< \alpha <  G$, we get
$$\alpha= \polter{1}{1} +\polter{-1}{d_\alpha (\varphi_\alpha(\alpha)) +\varphi^2_\alpha(\alpha)} \qquad
\mbox{\rm and} \qquad \alpha-2= \polter{-1}{1} +\polter{1}{d_\alpha (\varphi_\alpha (\alpha-2)) +\varphi^2_\alpha(\alpha)}.$$

We consider the rectangles
\begin{equation*}
\Omega_{I;\alpha} =I_\alpha \times [0,2-G),\quad
\Omega_{II;\alpha} =(\varphi_\alpha(\alpha),\alpha) \times (2-G,1],\quad
\Omega_{III;\alpha} = [\varphi_\alpha (\alpha-2),\alpha) \times [1,G),
\end{equation*}
and define
\begin{equation*}
\Omega_\alpha =\Omega_{I;\alpha} \cup \Omega_{II;\alpha} \cup \Omega_{III;\alpha}.
\end{equation*}

First, we consider $g< \alpha < 1$ and partition $\Omega_\alpha$,
up to a null-set, as $\Omega_{1;\alpha} \cup \cdots \cup \Omega_{7;\alpha}$, where
\begin{equation*}
\begin{split}
& \Omega_{1;\alpha} =\left[ \alpha-2,-\frac{1}{1+\alpha}\right) \times [0,2-G),\quad
\Omega_{2;\alpha} =\left( \frac{1}{1+\alpha},\alpha \right) \times [0,G),\\
& \Omega_{3;\alpha} =\left[ -\frac{1}{1+\alpha}, 0\right) \times [0,2-G),\quad
\Omega_{4;\alpha} = \left( 0,\frac{1}{1+\alpha}\right] \times [0,2-G) ,\\
& \Omega_{5;\alpha} = \left( 0,\frac{1}{1+\alpha}\right] \times [1,G), \quad
\Omega_{6;\alpha} = \left( \varphi_\alpha (\alpha),\frac{1}{1+\alpha}\right] \times [2-G,1), \\
& \Omega_{7;\alpha} = [\varphi_\alpha (\alpha-2),0) \times [1,G).
\end{split}
\end{equation*}

\begin{center}
\begin{figure}
\includegraphics[scale=1]{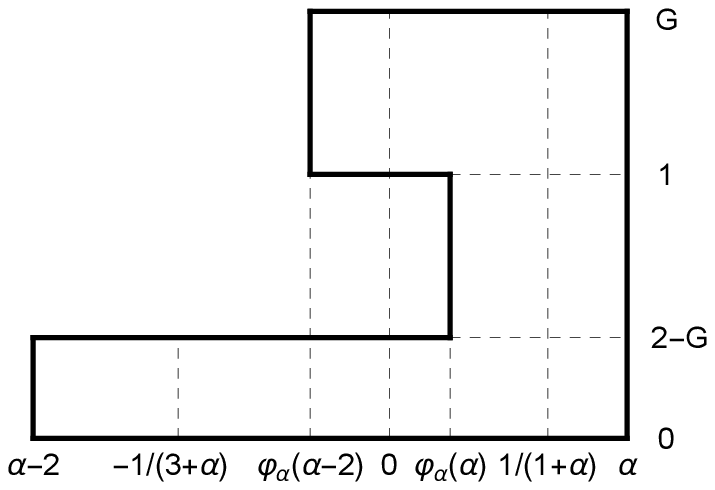}
\includegraphics[scale=1]{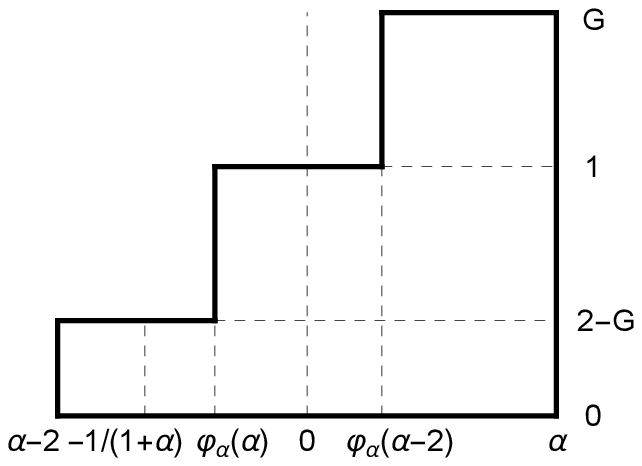}
\caption{\small The natural extension domains $\Omega_\alpha$ for
$g< \alpha \leq 1$ (left) and $1\leq \alpha < G$ (right)}\label{Figure1}
\end{figure}
\end{center}

\begin{lemma}\label{Lemma2}
Assuming $g< \alpha < 1$, the transformation $\Phi_\alpha$ maps $\Omega_\alpha$ one-to-one onto
$\Omega_\alpha$ up to a null-set. More precisely, denoting $2K+1=d_\alpha (\varphi_\alpha (\alpha))$ we have
\begin{equation*}
\begin{split}
& \Phi_\alpha (\Omega_{1;\alpha}) =\Omega_{III;\alpha}, \qquad  \Phi_\alpha (\Omega_{2;\alpha})=\Omega_{II;\alpha},
\qquad \Phi_\alpha(\Omega_{3;\alpha}) =I_\alpha \times \bigcup_{k\geq 1} \left[ \frac{1}{2k+1},\frac{1}{2k-1+G}\right),\\
& \Phi_\alpha (\Omega_{4;\alpha}) =I_\alpha \times \bigcup_{k\geq 1}
\left( \frac{1}{2k+3-G}, \frac{1}{2k+1}\right], \qquad
\Phi_\alpha (\Omega_{5;\alpha}) = I_\alpha \times \bigcup_{k\geq 1} \left( \frac{1}{2k+1+G}, \frac{1}{2k+2}\right],\\
& \Phi_\alpha (\Omega_{6;\alpha}) = I_\alpha \times \bigcup_{1\leq k\leq K-1}
\left( \frac{1}{2k+2},\frac{1}{2k+3-G}\right] \cup [ \alpha -2,\varphi_\alpha^2 (\alpha) )
\times \left( \frac{1}{2K+2},\frac{1}{2K+3-G} \right],\\
& \Phi_\alpha (\Omega_{7;\alpha})= I_\alpha \times \bigcup_{k\geq K+1}
\left[ \frac{1}{2k+2},\frac{1}{2k+3-G}\right) \cup
[ \varphi_\alpha^2 (\alpha),\alpha ) \times\left[ \frac{1}{2K+2},\frac{1}{2K+3-G}\right) .
\end{split}
\end{equation*}
\end{lemma}

\begin{proof}
The following situations occur:

(i) $(x,y)\in \Omega_{1;\alpha}$, so $d_\alpha (x)=1$, $e(x)=-1$. The range of
$\Phi_\alpha (x,y)=( \frac{1}{| x |} -1,\frac{1}{1-y})$
is exactly $[ \varphi_\alpha (\alpha-2),\alpha) \times [  1,G ) =\Omega_{III;\alpha}$.

(ii) $(x,y)\in \Omega_{2;\alpha}$, so $d_\alpha(x)=1$, $e(x)=1$. The range of
$\Phi_\alpha (x,y)=( \frac{1}{x}-1,\frac{1}{1+y})$ is exactly $\Omega_{II;\alpha}$.

(iii) $(x,y)\in \Omega_{3;\alpha}$, so $d_\alpha(x)=2k+1 \geq 3$, $e(x) =-1$, and
$\Phi_\alpha ( x,y) =( \frac{1}{| x|}-(2k+1), \frac{1}{2k+1-y})$
covers the rectangle $I_\alpha \times [ \frac{1}{2k+1},\frac{1}{2k-1+G})$
when $(x,y)$ covers $\langle -2k-1\rangle_\alpha  \times
[0,2-G)$.

(iv) $(x,y)\in \Omega_{4;\alpha}$, so $d_\alpha (x)=2k+1 \geq 3$, $e(x) =1$, and
$\Phi_\alpha ( x,y) =( \frac{1}{x}-(2k+1), \frac{1}{2k+1+y})$
covers the rectangle $I_\alpha \times ( \frac{1}{2k+3-G},\frac{1}{2k+1}]$
when $(x,y)$ covers $\langle 2k+1\rangle_\alpha \times [0,2-G)$.

(v) $(x,y)\in \Omega_{5;\alpha}$, so $d_\alpha (x)=2k+1 \geq 3$, $e(x) =1$, and
$\Phi_\alpha ( x,y) =( \frac{1}{x}-(2k+1), \frac{1}{2k+1+y})$
covers the rectangle $I_\alpha \times ( \frac{1}{2k+1+G},\frac{1}{2k+2}]$
when $(x,y)$ covers $\langle 2k+1 \rangle_\alpha \times [1,G)$.

(vi) $(x,y)\in \Omega_{6;\alpha}$, so $d_\alpha (x)=2k+1\geq 3$, $e(x) =1$.
Since $d_\alpha (\varphi_\alpha(\alpha))=2K+1$, we have
$\varphi_\alpha^2 (\alpha)= \frac{\alpha}{1-\alpha}
-(2K+1)$.
Then $\Phi_\alpha (x,y) =( \frac{1}{x}-(2k+1),\frac{1}{2k+1+y})$ covers
the rectangle $I_\alpha \times ( \frac{1}{2k+2},\frac{1}{2k+3-G}]$
when $(x,y)$ covers $\langle 2k+1\rangle_\alpha \times [2-G,1)$ and $1\leq k <K$,
and if $k=K$ it covers the rectangle
$[ \alpha -2,\frac{\alpha}{1-\alpha}-2K-1) \times ( \frac{1}{2K+2},\frac{1}{2K+3-G}]$
when $(x,y)$ covers $( \varphi_\alpha(\alpha),\frac{1}{2K-1+\alpha}] \times [2-G,1)$.

(vii) $(x,y)\in \Omega_{7;\alpha}$, so $d_\alpha (x)=2\ell+1 \geq 5$, $e(x)=-1$.
Let $L$ such that $d_\alpha (| \varphi_\alpha (\alpha-2)| )=d_\alpha ( \frac{1-\alpha}{2-\alpha})
=2L+1$. By \eqref{2.2} we have $L=K+1$. Then
$\Phi_\alpha (x,y) =( \frac{1}{| x|}-(2\ell +1),y^\prime=\frac{1}{2\ell +1-y})$ covers
the rectangle $I_\alpha \times [ \frac{1}{2\ell},\frac{1}{2\ell+1-G})$
when $(x,y)$ covers the rectangle $\langle 2\ell+1 \rangle_\alpha \times [1,G)$
and $\ell > L=K+1$, and if $\ell=K+1$ it covers the rectangle
$[ \frac{\alpha}{1-\alpha} -2K-1,\alpha ) \times
[ \frac{1}{2K+2} ,\frac{1}{2K+3-G})$
when $(x,y)$ covers $[ \varphi_\alpha (\alpha-2),-\frac{1}{2K+3+\alpha}) \times [1,G)$.

The map $\Phi_\alpha$ is one-to-one and onto on $\Omega_\alpha$ up to a null-set because the interior sets
$\mathring{\Omega}_{1;\alpha},\ldots ,\mathring{\Omega}_{7;\alpha}$ and their images $\Phi_\alpha(\mathring{\Omega}_1),\ldots,
\Phi_\alpha (\mathring{\Omega}_7)$ are disjoint.
\end{proof}

When $1< \alpha < G$, we consider
\begin{equation*}
\begin{split}
& \Omega_{1;\alpha} =\left[ \alpha-2,-\frac{1}{1+\alpha}\right) \times [0,2-G) ,\quad
\Omega_{2;\alpha} =\left( \frac{1}{1+\alpha},\alpha \right) \times [0,1) ,\\
& \Omega_{3;\alpha} = \left[ -\frac{1}{1+\alpha},0\right) \times [0,2-G),\quad
\Omega_{4;\alpha} = \left( 0,\frac{1}{1+\alpha} \right] \times [0,1),\quad
\Omega_{5;\alpha} = (\varphi_\alpha (\alpha),0) \times (2-G,1] .
\end{split}
\end{equation*}

When $\frac{1}{1+\alpha} < \varphi_\alpha (\alpha-2) <\alpha$, we take $\Omega_{6;\alpha}=[\varphi_\alpha (\alpha-2),\alpha) \times [1,G)$
and partition $\Omega_\alpha$ modulo a null-set as $\Omega_{1;\alpha} \cup \cdots \cup \Omega_{6;\alpha}$.
When $\varphi_\alpha (\alpha-2) \leq \frac{1}{1+\alpha}$, we take
$\Omega_{6;\alpha} =[ \varphi_\alpha (\alpha-2),\frac{1}{1+\alpha}] \times [1,G)$,
$\Omega_{7;\alpha} = ( \frac{1}{1+\alpha} ,\alpha ) \times [1,G)$,
and partition $\Omega_\alpha$ modulo a null-set as $\Omega_{1;\alpha} \cup \cdots \cup \Omega_{7;\alpha}$.

We write $d_\alpha (\varphi_\alpha (\alpha-2))=2L+1$, $L\geq 0$, and
$d_\alpha (-\varphi_\alpha (\alpha))=2K+1$, $K\geq 1$, so that
\begin{equation*}
\frac{2-\alpha}{\alpha-1} -(2L+1)=\varphi_\alpha^2 (\alpha-2) \qquad \mbox{\rm and} \qquad
\frac{\alpha}{\alpha-1} -(2K+1) =\varphi_\alpha^2(\alpha).
\end{equation*}
By \eqref{2.2} we have $K=L+1$.
Similarly to the proof of Lemma \ref{Lemma2} we find
\begin{equation}\label{2.3}
\begin{split}
& \Phi_\alpha (\Omega_{1;\alpha})=\Omega_{III;\alpha}, \quad
\Phi_\alpha (\Omega_{2;\alpha}) = (\varphi_\alpha(\alpha),\alpha) \times \left( \frac{1}{2},1\right],\\
& \Phi_\alpha (\Omega_{3;\alpha}) = I_\alpha \times \bigcup_{k\geq 1} \left[ \frac{1}{2k+1},\frac{1}{2k-1+G}\right),\quad
 \Phi_\alpha (\Omega_{4;\alpha}) =I_\alpha \times \bigcup_{k\geq 1} \left( \frac{1}{2k+2},\frac{1}{2k+1}\right],\\
& \Phi_\alpha (\Omega_{5;\alpha}) =( \varphi_\alpha^2 (\alpha),\alpha ) \times\left( \frac{1}{2K-1+G},\frac{1}{2K} \right]
\cup I_\alpha \times \bigcup_{k>K} \left( \frac{1}{2k-1+G},\frac{1}{2k}\right] .
\end{split}
\end{equation}

\begin{lemma}\label{Lemma3}
Assuming $1< \alpha < G$, the transformation $\Phi_\alpha$ maps $\Omega_\alpha$ one-to-one onto
$\Omega_\alpha$ up to a null-set.
\end{lemma}

\begin{proof}
We consider the two possible situations mentioned above:

(i) $\frac{1}{1+\alpha} < \varphi_\alpha (\alpha-2) <\alpha$,
which corresponds to $\frac{1}{2}(\sqrt{13}-1) <\alpha <G$ and yields
$L=0$, $K=1$, and $\varphi_\alpha^2 (\alpha-2)=\frac{3-2\alpha}{\alpha-1}$.
We have $\Phi_\alpha (\Omega_{6;\alpha}) =( \varphi_\alpha (\alpha),\varphi_\alpha^2 (\alpha-2) ]
\times ( 2-G,\frac{1}{2}]$ and
\begin{equation*}
\Phi_\alpha (\Omega_{5;\alpha} \cup \Omega_{6;\alpha} \cup \Omega_{2;\alpha}) = [\varphi_\alpha (\alpha),\alpha) \times (2-G,1]
\cup I_\alpha \times \bigcup_{k\geq 2} \left( \frac{1}{2k+1+G},\frac{1}{2k+2}\right] ,
\end{equation*}
which lead in conjunction with \eqref{2.3} to the desired result.

(ii) $\varphi_\alpha (\alpha-2) \leq \frac{1}{1+\alpha}$, which yields $L=K-1\geq 1$.
In this case we find
\begin{equation*}
\begin{split}
\Phi_\alpha (\Omega_{6;\alpha}) & = [ \alpha -2, \varphi_\alpha^2 (\alpha-2)] \times
\left( \frac{1}{2L+1+G},\frac{1}{2L+2}\right] \cup I_\alpha \times
\bigcup_{1\leq \ell <L} \left( \frac{1}{2\ell+1+G},\frac{1}{2\ell +2}\right] ,
 \\
\Phi_\alpha (\Omega_{7;\alpha}) & = (\varphi_\alpha (\alpha),\alpha ) \times \left( 2-G,\frac{1}{2}\right] .
\end{split}
\end{equation*}
Employing also $K=L+1$, $\varphi_\alpha^2 (\alpha)=\varphi_\alpha^2 (\alpha -2)$ and \eqref{2.3},
we establish the desired result.
\end{proof}


\begin{remark}\label{Remark4}
When $\alpha \in \{ g,1,G\}$, the analogues of Lemmas \ref{Lemma2} and \ref{Lemma3}
still hold and are checked in a similar way.

When $\alpha =g$, corresponding to $\alpha \searrow g$, we take
$\Omega_g=[g-2,g) \times [0,1-g) \cup[\frac{g-1}{2-g},g) \times [1,G)$,
partitioned into $\Omega_{1;g}=[g-2,-g) \times [0,2-g)$,
$\Omega_{3;g}=[-g,0) \times [0,2-G)$, $\Omega_{4;g}=(0,g] \times [0,2-G)$,
$\Omega_{5;g}=(0,g] \times [1,G)$, $\Omega_{7;g}=(\frac{g-1}{2-g},0) \times [1,G)$.

When $\alpha =G$, corresponding to $\alpha \nearrow G$, we take
$\Omega_G=[G-2,G) \times [0,1)$,
partitioned into $\Omega_{2;G}=(2-G,G) \times [0,1)$;
$\Omega_{3;G}=[G-2,0) \times [0,2-G)$, $\Omega_{4;G}=(0,2-G] \times [0,1)$,
$\Omega_{5;G}=[G-2,0) \times (2-G,1]$.

When $\alpha =1$ we can take,
as in Lemma \ref{Lemma2}, $\Omega_1 =[-1,1) \times [0,2-G) \cup [0,1)\times [2-G,G)$,
partitioned into
$\Omega_{1;1}=[-1,-\frac{1}{2}) \times [0,2-G)$,
$\Omega_{2;1}=(\frac{1}{2},1) \times [0,G)$,
$\Omega_{3;1}=[-\frac{1}{2},0) \times [0,2-G)$,
$\Omega_{4;1}=(0,\frac{1}{2}] \times [0,2-G)$,
$\Omega_{5;1}=(0,\frac{1}{2}] \times [1,G)$;
$\Omega_{6;1} =(0,\frac{1}{2}] \times [2-G,1)$.


\end{remark}

\newpage

\section{The $\alpha$-OCF expansions and growth of denominators}\label{alpha_ocf}
The next two lemmas will be helpful in estimating the rate of growth of
the denominators of convergents of $x$. For each $(x,y)\in \Omega_\alpha \setminus \Q^2$
we denote $(x_k,y_k)= \Phi^k_\alpha (x,y)$.

\begin{lemma}\label{Lemma5}
Assume $g< \alpha \leq 1$. For every $(x,y)\in \Omega_\alpha \setminus \Q^2$, at least one of the next
five inequalities holds:
\begin{equation*}
\begin{split}
& y_0 \leq 2-G <\frac{1}{\sqrt{2}}, \quad
y_0y_1 \leq \frac{G}{5-G} <\frac{1}{2}, \quad
y_0y_1y_2 \leq \frac{1}{3} < \frac{1}{2\sqrt{2}} ,\\
& y_0y_1y_2y_3 \leq \frac{G}{7+4G} < \frac{1}{4},\quad
\mbox{or} \quad y_0 y_1 y_2 y_3 y_4  \leq \frac{1}{5G-2} < \frac{1}{4\sqrt{2}} .
\end{split}
\end{equation*}
\end{lemma}

\begin{proof}
The following five situations can occur:

(a) $(x,y)\in\Omega_{I;\alpha}$, when we clearly have $y<2-G <\frac{1}{\sqrt{2}}$.

(b) $\varphi_\alpha(\alpha) <x\leq \frac{1}{1+\alpha}$ and $2-G\leq y\leq 1$. Then we have $d_\alpha (x)\geq 3$, $e(x)=1$, so
$y_1\leq \frac{1}{3+y}$ and $yy_1 \leq \frac{y}{3+y} \leq \frac{1}{4}$.
When $\frac{1}{1+\alpha} <x<\alpha$ and $2-G \leq y\leq 1$, we have $y_1=\frac{1}{1+y}$
and $\varphi_\alpha (\alpha)<x_1 =\frac{1}{x}-1 <\alpha$, so $y_2 \leq \frac{1}{1+y_1} =\frac{1+y}{2+y}$ and
$yy_1y_2 \leq \frac{y}{2+y} \leq \frac{1}{3}$.

(c) $0<x\leq \frac{1}{1+\alpha}$ and $1\leq y\leq G$. Then we have $d_\alpha (x)\geq 3$, $e(x)=1$, so $y_1\leq \frac{1}{3+y}$ and
$yy_1 \leq \frac{y}{3+y} \leq \frac{G}{3+G} <\frac{1}{2}$.

(d) $\varphi_\alpha (\alpha-2) \leq x<0$ and $1\leq y\leq G$. Then we have $d_\alpha (x)\geq 5$, $e(x)=-1$, so $yy_1\leq \frac{y}{5-y}
\leq \frac{G}{5-G}$.

(e) When $\frac{1}{1+\alpha} <x<\alpha$ and $1\leq y\leq G$, we have $(x_1,y_1) \in \Omega_{II;\alpha}$, $y_1 = \frac{1}{1+y}$,
$yy_1 =\frac{y}{1+y} \leq \frac{G}{1+G}$, $y_2 =\frac{1}{d_2+y_1} \leq \frac{1}{1+y}$, and one of the following two situations holds:
(e.1) $0\leq y\leq 1$, when $yy_1 y_2 \leq \frac{y}{2+y} \leq \frac{1}{3}$;
(e.2) $1\leq y\leq G$, when either (e.2.1) $\varphi_\alpha (\alpha) < x_1 < \frac{1}{1+\alpha}$,
so $y_2 \leq \frac{1}{3+y_1}$ and $yy_1y_2 \leq \frac{y}{4+3y}\leq \frac{G}{4+3G}$,
or (e.2.2) $\frac{1}{1+\alpha} <x_1 <\alpha$, so $y_2= \frac{1}{1+y_1} = \frac{1+y)}{2+y}$ and
$yy_1y_2= \frac{y}{2+y}$. In sub-case (e.2.2) two situations can again occur:
(e.2.2.1) $\varphi_\alpha (\alpha) <x_2 < \frac{1}{1+\alpha}$, so $y_3 \leq \frac{1}{3+y_2}$ and
$yy_1y_2y_3 \leq \frac{y}{7+4y} \leq \frac{G}{7+4G}$, or
(e.2.2.2) $\frac{1}{1+\alpha} <x_2 <\alpha$, so $y_3=\frac{1}{1+y_2}$ and
$yy_1y_2y_3 =\frac{y}{3+2y}$. Furthermore, since $\varphi_\alpha (\alpha) < x_3 <\alpha$
we have $y_4=\frac{1}{1+y_3} =\frac{3+2y}{5+3y}$ and
$y y_1y_2y_3y_4= \frac{y}{5+3y} \leq \frac{G}{5+3G}= \frac{1}{5G-2}$, concluding the proof.
\end{proof}

\begin{lemma}\label{Lemma6}
Assume $1\leq \alpha < G$. For every $(x,y)\in \Omega_\alpha \setminus \Q^2$,
at least one of the next five inequalities holds:
\begin{equation*}
\begin{split}
& y_0 \leq 2-G <\frac{1}{\sqrt{2}}, \quad
y_0y_1 \leq \frac{G}{3+G} <\frac{1}{2}, \quad
y_0y_1y_2 \leq \frac{1}{3} < \frac{1}{2\sqrt{2}} ,\\
& y_0y_1 y_2 y_3 \leq \frac{G}{5+2G} < \frac{1}{4},\quad
\mbox{or} \quad y_0 y_1 y_2 y_3 y_4  \leq \frac{1}{5G-2} < \frac{1}{4\sqrt{2}} .
\end{split}
\end{equation*}
\end{lemma}

\begin{proof}
We can assume $2-G \leq y \leq G$. When $0< x\leq \frac{1}{1+\alpha}$ and
$2-G \leq y\leq 1$, we have $d_\alpha (x) \geq 3$,
$e(x)=1$ and $yy_1 \leq \frac{y}{3+y} \leq \frac{1}{4}$.
Four more situations can occur:

(a) $\frac{1}{1+\alpha} <x<\alpha$ and $2-G\leq y\leq 1$, Then we have
$y_1= \frac{1}{1+y}$ and $-\frac{1}{1+\alpha} < \frac{1}{\alpha} -1 =\varphi_\alpha (\alpha)
<x_1 = \frac{1}{x}-1 <\alpha$. Two sub-cases can occur: (a.1) $x_1 >0$, so
$y_2 \leq \frac{1}{1+y_1}=\frac{1+y}{2+y}$ and
$yy_1y_2 =\frac{y}{2+y} \leq \frac{1}{3}$;
(a.2) $x_1 <0$, so $d_\alpha (x_1) \geq 3$, $e(x_1)=-1$ and
$y_2 \leq \frac{1}{3-y_1} =\frac{1+y}{2+3y}$,
$yy_1y_2 \leq \frac{y}{2+3y} \leq \frac{1}{5}$.

(b) $\varphi_\alpha (\alpha) <x<0$ and $2-G \leq y\leq 1$. Then we have $d_\alpha (x) \geq 3$,
$e(x)=-1$, so $y_1 \leq \frac{1}{3-y}$. Two sub-cases can occur:
(b.1) $x\geq  -\frac{1}{3+\alpha}$, so $d_\alpha (x) \geq 5$, $e(x)=-1$, and
$y_1 \leq \frac{1}{5-y}$, $yy_1 \leq \frac{y}{5-y} \leq \frac{1}{4}$;
(b.2) $-\frac{1}{1+\alpha} \leq \varphi_\alpha (\alpha) < x < -\frac{1}{3+\alpha}$.
Then $d_\alpha (x)=3$, $e(x)=-1$, $y_1=\frac{1}{3-y}$ and
$\frac{3-2\alpha}{\alpha-1} < x_1 =\frac{1}{| x|} -3 <\alpha$.
If $x_1 >0$, then $y_2 \leq \frac{1}{1+y_1} =\frac{3-y}{4-y}$ and
$yy_1y_2 =\frac{y}{4-y} \leq \frac{1}{3}$. On the other hand,
since $-\frac{1}{1+\alpha} \leq \frac{3-2\alpha}{\alpha-1}$, notice that when
$\frac{3-2\alpha}{\alpha-1} < x_1 <0$ we must have $d_\alpha (x_1)\geq 3$, $e(x_1)=-1$,
and so $y_2 \leq \frac{1}{3-y_1} =\frac{3-y}{8-3y}$ and
$yy_1y_2 \leq \frac{y}{8-3y} \leq \frac{1}{5}$.

(c) When $\varphi_\alpha (\alpha-2)\leq x\leq \frac{1}{1+\alpha}$ and $1\leq y\leq G$, we have
$d_\alpha (x)\geq 3$, $e(x)=1$, so $y_1 \leq \frac{1}{3+y}$ and
$yy_1 \leq \frac{y}{3+y} \leq \frac{G}{3+G} < \frac{1}{2}$.

(d) When $\varphi_\alpha (\alpha-2) < \frac{1}{1+\alpha} < x <\alpha$ and $1<y<G$, we have
$y_1 = \frac{1}{1+y}$ and $\frac{1}{\alpha}-1 < x_1= \frac{1}{x}-1 < \alpha$.
Three sub-cases can occur: (d.1) $\varphi_\alpha (\alpha) < x_1 <0$, so $d_\alpha (x_1) \geq 3$, $e(x_1)=-1$,
$y_2 \leq \frac{1}{3-y_1} = \frac{1+y}{2+3y}$,
$yy_1y_2 \leq \frac{y}{2+3y} \leq \frac{G}{2+3G} < \frac{1}{3}$;
(d.2) $0< x_1 < \frac{1}{1+\alpha}$, so $y_2 \leq \frac{1}{3+y_1} \leq
\frac{1}{3-y_1}$ and $yy_1y_2 \leq \frac{G}{2+3G}$ as in (d.1);
(d.3) $\frac{1}{1+\alpha} < x_1 < \alpha$, so $y_2=\frac{1}{1+y_1} =\frac{1+y}{2+y}$
and $\frac{1}{\alpha} -1 < x_2 = \frac{1}{x_1}-1 <\alpha$.
Two more sub-cases can occur, as follows:
(d.3.1) $\frac{1}{\alpha} -1 < x_2 < \frac{1}{1+\alpha}$, when as in cases (d.1) and (d.2) above we get
$y_3 \leq \frac{1}{3-y_2}$ and $yy_1y_2y_3 \leq \frac{y}{5+2y} \leq \frac{G}{5+2G}$;
(d.3.2) $\frac{1}{1+\alpha} < x_2 <\alpha$, so $y_3= \frac{1}{1+y_2} =\frac{2+y}{3+2y}$ and
$\frac{1}{\alpha}-1 <x_3 =\frac{1}{x_2}-1 < \alpha$. Two sub-cases can occur here:
(d.3.2.1) $\frac{1}{\alpha}-1 < x_3 < \frac{1}{1+\alpha}$, so
$y_4 \leq \frac{1}{3-y_3} =\frac{3+2y}{7+5y}$ and
$yy_1y_2y_3y_4 =\frac{y}{7+5y} \leq \frac{G}{7+5G} < \frac{G}{5+3G}$;
(d.3.2.2) $\frac{1}{1+\alpha} < x_3 < \alpha$, so
$y_4= \frac{1}{1+y_3} =\frac{3+2y}{5+3y}$ and
$yy_1y_2y_3y_4 =\frac{y}{5+3y} \leq \frac{G}{5+3G}$, concluding the proof.
\end{proof}

\begin{remark}\label{Remark7}
Lemma \ref{Lemma5} also works for $\alpha=g$ and Lemma \ref{Lemma6} for $\alpha=G$.
\end{remark}

With $d_\alpha$ as in \eqref{1.2} and $e(x)=\operatorname{sign}(x)$, we define
\begin{equation*}
d_i(x;\alpha) = d_\alpha (\varphi_\alpha^{i-1}(x)),\quad e_i(x;\alpha)=e(\varphi_\alpha^{i-1}(x))
\qquad \mbox{\rm if $\varphi_\alpha^{i-1}(x)\neq 0$,}
\end{equation*}
and also let $d_i(x;\alpha)=\infty$, $e_i(x;\alpha)=0$ if $\varphi_\alpha^{i-1}(x)=0$.
We also define
\begin{equation*}
\omega_{\alpha,i} (x) :=e_i (x;\alpha) d_i(x;\alpha)\in 2\Z -1 ,\qquad i\geq 1,
\end{equation*}
and consider the cylinder sets
\begin{equation*}
\begin{split}
\langle \omega_1,\ldots,\omega_n \rangle_\alpha & :=
\langle \omega_1 \rangle_\alpha \cap \varphi_\alpha^{-1} (\langle \omega_2 \rangle_\alpha)
\cap \ldots \cap \varphi_\alpha^{-(n-1)} (\langle \omega_n \rangle_\alpha ) \\
& =\{ x\in I_\alpha : \omega_{\alpha,1} (x)=\omega_1, \ldots, \omega_{\alpha,n}(x)=\omega_n \} ,
\end{split}
\end{equation*}
with $\omega_i$ positive or negative odd integers.

Define $p_n=p_n(x;\alpha)$ and $q_n=q_n (x;\alpha)$ by
\begin{equation*}
\begin{split}
& p_{-1} =1,\quad p_0=0,\quad p_1 =e(x) ,\\
& q_{-1} =0,\quad q_0  =1,\quad q_1 =d_{\alpha} (x) ,\\
& \left( \begin{matrix} p_n & q_n \\ p_{n-1} & q_{n-1} \end{matrix}\right) =
\left( \begin{matrix} d_n & e_n \\ 1 & 0 \end{matrix} \right)
\left( \begin{matrix} p_{n-1} & q_{n-1} \\ p_{n-2} & q_{n-2} \end{matrix} \right) ,
\end{split}
\end{equation*}
with $d_n=d_n (x;\alpha)$ and $e_n =e_n (x;\alpha)$.
Whenever $x\in I_\alpha$ with $\varphi_\alpha^i (x)\neq 0$, $i=0,\ldots,n-1$, we have
\begin{equation*}
x=\polter{e_1}{d_1}+\polter{e_2}{d_2}+\cdots+\polter{e_n}{d_n+\varphi_\alpha^n (x)} \qquad\mbox{\rm and} \qquad
\frac{p_n}{q_n} =\polter{e_1}{d_1}+\polter{e_2}{d_2}+\cdots+\polter{e_n}{d_n} .
\end{equation*}
The following standard properties hold:
\begin{equation}\label{3.1}
p_{n-1}(x;\alpha)q_n (x;\alpha) -p_n (x;\alpha) q_{n-1}(x;\alpha) =(-1)^n
e_1 (x;\alpha)\cdots e_n (x;\alpha) ,
\end{equation}
\begin{equation}\label{3.2}
x= \frac{p_n(x;\alpha) +\varphi_\alpha^n (x) p_{n-1}(x;\alpha)}{q_n(x;\alpha)+\varphi_\alpha^n (x) q_{n-1}(x;\alpha)},
\end{equation}
\begin{equation}\label{3.3}
\varphi_\alpha^n (x) = \frac{q_n(x;\alpha)x-p_n (x;\alpha)}{-q_{n-1}(x;\alpha)x +p_{n-1}(x;\alpha)} ,
\end{equation}
\begin{equation}\label{3.4}
\Phi_\alpha^n (x,0)=\left( \varphi_\alpha^n (x),\frac{q_{n-1}(x;\alpha)}{q_n (x;\alpha)} \right) ,
\qquad \forall n\geq 0.
\end{equation}
Since $d_\alpha(x)+e(x)y>0$ for all $(x,y)\in\Omega_\alpha$, the expression for $\Phi_\alpha$
in \eqref{1.3} and formula \eqref{3.4} show
that $q_n(x;\alpha)>0$ for every $x\in \II_\alpha$ and $n\geq 1$.

\begin{proposition}\label{Prop8}
For every $x\in \II_\alpha$ and $n\in \N$ we have
\begin{itemize}
\item[(i)]
$\displaystyle \min \left\{ \frac{q_n(x;\alpha)}{q_{n+1}(x;\alpha)},
\frac{q_n(x;\alpha)}{q_{n+2}(x;\alpha)},
\frac{q_n(x;\alpha)}{q_{n+3}(x;\alpha)},
\frac{q_n(x;\alpha)}{q_{n+4}(x;\alpha)},
\frac{q_n(x;\alpha)}{q_{n+5}(x;\alpha)} \right\} \leq
\left( \frac{1}{5G-2} \right)^{1/5}$.
\item[(ii)]
$\displaystyle
q_n (x;\alpha) \geq q B^n$, where $B= (5G-2)^{1/5} \approx 1.43524 > \sqrt{2}$ and
$q=(\frac{g}{B})^4 \approx 0.03438$.
\end{itemize}
\end{proposition}

\begin{proof}
(i) This follows from Lemmas \ref{Lemma5} and \ref{Lemma6}, and identity \eqref{3.4}.

(ii) Denote $A_1=2-G$, $A_2 =G/(5-G)$, $A_3= 1/3$,
$A_4= G/(5+2G)$, $A_5= 1/(5G-2)$.
Fix $n\geq 5$. By (i) for every $x$ there exist $n_0=n_0(x) \in\{ n-4,n-3,n-2,n-1,n\}$
and non-negative integers $\alpha_1,\ldots,\alpha_5$ such that
$\alpha_1+2\alpha_2 + 3\alpha_3+4\alpha_4+5\alpha_5 =n_0$ and
\begin{equation*}
\frac{1}{q_{n_0}(x;\alpha)} =\frac{q_0(x;\alpha)}{q_{n_0}(x;\alpha)} \leq
A_1^{\alpha_1} \cdots A_5^{\alpha_5} \leq A^{n_0},
\end{equation*}
where $A:=\max\{ A_1,A_2^{1/2},A_3^{1/3},A_4^{1/4},A_5^{1/5}\} =A_5^{1/5}
=(5G-2)^{-1/5}$. Since $\frac{q_k (x;\alpha)}{q_{k+1}(x;\alpha)} \leq G$, we find that
$q_n(x;\alpha) \geq \min \{ A^{-n},gA^{-n+1},g^2 A^{-n+2},g^3 A^{-n+3},g^4 A^{-n+4}\}
=qA^{-n}=qB^n$.
\end{proof}

\begin{lemma}\label{Lemma9}
For every $x\in \II_\alpha$ and $n\in \N$ we have
\begin{equation}\label{3.5}
2G q_n (x;\alpha) \geq q_n(x;\alpha)+\varphi_\alpha^n (x) q_{n-1}(x;\alpha) \geq C_\alpha q_{n}(x;\alpha),
\end{equation}
where
\begin{equation*}
C_{\alpha} =\min_{(u,v)\in \Omega_\alpha} (1+xy) \geq 2(\sqrt{5}-2) >0.
\end{equation*}
\end{lemma}

\begin{proof}
Inequality \eqref{3.5} follows directly from \eqref{3.4}. The lower bound in $C_\alpha$ is
derived using the particular shape of $\Omega_\alpha$ as follows:
\begin{equation*}
\begin{split}
C_\alpha & =1+ \begin{cases}
\min\big\{ (2-G)(\alpha-2),\frac{G(\alpha-1)}{2-\alpha}\big\} &
\mbox{\rm if $g\leq \alpha \leq 1$} \\
\min \big\{ (2-G)(\alpha-2),\frac{1-\alpha}{\alpha}\big\} &
\mbox{\rm if $1\leq g\leq G$,}
\end{cases}  \\
& = \begin{cases}
1+ (2-G)(\alpha-2) &
\mbox{\rm if $g\leq \alpha \leq 1$} \\
1+\min \big\{ (2-G)(\alpha-2),\frac{1-\alpha}{\alpha}\big\}
=\frac{1}{\alpha_*}-1 &
\mbox{\rm if $1\leq g\leq G$,}
\end{cases}
\end{split}
\end{equation*}
where $\alpha_* =\frac{1}{2} (-G+1+\sqrt{G^2+2G+5})$ and $C_\alpha \geq 2(\sqrt{5}-2)$.
\end{proof}

\begin{proposition}\label{Prop10}
For every $\alpha \in [g,G]$, every $x\in \II_\alpha$,
and every $n\in\N$, there exist universal constants $c_1,c_2>0$ such that
\begin{equation*}
\frac{1}{q_n(x;\alpha) q_{n+1}(x;\alpha)} \leq
\left| x-\frac{p_n(x;\alpha)}{q_n(x;\alpha)} \right|
\leq \frac{c_1}{q_n^2 (x;\alpha)} \leq c_2 C^{-n},
\end{equation*}
where $C=( 5G-2)^{2/5} \approx2.05993$.
\end{proposition}

\begin{proof}
The second inequality follows from \eqref{3.2}, \eqref{3.1} and Lemma \ref{Lemma9}.
The third inequality follows from Proposition \ref{Prop8}.

To prove the first inequality, denote $d_{n+1}:=d_{n+1} (x;\alpha)$,
$p_n:=p_n(x;\alpha)$, $q_n=q_n (x;\alpha)$, and $u=\varphi_\alpha^n (x)$.
The equality $d_{n+1}=2[ \frac{1}{2| u|} +\frac{1-\alpha}{2}]+1$ shows that
$\frac{1}{| u|} < d_{n+1}+\alpha$. In conjunction with equality \eqref{3.2} we infer
\begin{equation*}
\begin{split}
\left| x-\frac{p_n}{q_n} \right| & = \frac{| u|}{q_n | q_n +uq_{n-1}|}
=\frac{1}{q_n^2 \big( \frac{1}{| u|} +\frac{q_{n-1}}{q_n}\big)}
= \frac{1}{q_n \big( q_{n-1}+\frac{q_n}{| u|}\big)}  \\
& > \frac{1}{q_n ( q_{n-1}+(d_{n+1}+\alpha)q_n)} =
\frac{1}{q_n} \cdot \frac{1}{q_{n+1}+(1-e_{n+1})q_{n-1}+\alpha q_n}
> \frac{1}{q_n q_{n+1}} .  \qedhere
\end{split}
\end{equation*}
\end{proof}

\section{Ergodic properties of the map $\varphi_\alpha$}\label{alpha_ergodic}
In this section we explicitly describe the $\varphi_\alpha$-invariant Lebesgue absolutely continuous probability measure $\nu_\alpha$
and the natural extension of $(I_\alpha,\BB_{I_\alpha},\nu_\alpha,\varphi_\alpha)$, show that
$\varphi_\alpha$ is an exact endomorphism, and compute its entropy with respect to $\nu_\alpha$.

\begin{lemma}\label{Lemma11}
The measure $d\mu_\alpha =(1+xy)^{-2} dx dy$ is $\Phi_\alpha$-invariant on $\Omega_\alpha$
and $\mu_\alpha (\Omega_\alpha)=3\log G$.
\end{lemma}

\begin{proof}
If $u= \frac{e}{x}-d$ and
$v=\frac{1}{d+ey}$, then
\begin{equation*}
\frac{1}{(1+xy)^2} \cdot \frac{\partial (x,y)}{\partial (u,v)}  =
\frac{1}{(1+uv)^2} .
\end{equation*}
The measure $\mu_\alpha$ is $\Phi_\alpha$-invariant since $\Phi_\alpha (x,y)= ( \frac{e}{x}-d,\frac{1}{d+e y})$
for every $(x,y) \in (\operatorname{int} \langle \omega \rangle_\alpha \times [0,G))
\cap \Omega_\alpha$, where $\omega =ed\in 2\Z -1$. A direct calculation gives
$\mu_\alpha (\Omega_\alpha)=3\log G$.
\end{proof}

Consider the projection $\pi_\alpha :\Omega_\alpha \rightarrow I_\alpha$, $\pi_\alpha (x,y)=x$,
and the section sets $R_\alpha (x)=\pi_\alpha^{-1}(x)$.
Consider also the $\Phi_\alpha$-invariant probability measure
$\tilde{\mu}_\alpha =(3\log G)^{-1} \mu_\alpha$.
The probability measure $\nu_\alpha$ on $I_\alpha$ defined by
\begin{equation*}
\nu_\alpha (E)=\tilde{\mu}_\alpha (\{ (x,y)\in\Omega_\alpha: x\in E\})=
\tilde{\mu}_\alpha \bigg( \bigcup\limits_{x\in E} \{ x\} \cup R_\alpha (x)\bigg),
\qquad E\in \BB_{I_\alpha},
\end{equation*}
is $\varphi_\alpha$-invariant and Lebesgue absolutely continuous since $\Phi_\alpha$ is a skew-shift over $\varphi_\alpha$.

\begin{cor}\label{Cor12}
$d\nu_\alpha =h_\alpha d\lambda$ with density $h_\alpha$ as follows:
\begin{equation*}
\begin{split}
h_\alpha (x) & = \frac{1}{3\log G} \cdot
\begin{cases} \frac{1}{x+G+1} & \mbox{if $x\in \big[ \alpha-2,\frac{\alpha-1}{2-\alpha}\big)$} \\
\frac{1}{x+G+1}+\frac{1}{x+G-1}-\frac{1}{x+1} &
\mbox{if $x\in \big[ \frac{\alpha-1}{2-\alpha},\frac{1-\alpha}{\alpha}\big)$} \\
\frac{1}{x+G-1} & \mbox{if $x\in \big[ \frac{1-\alpha}{\alpha}, \alpha \big)$}
\end{cases}
\qquad \mbox{when $g\leq \alpha \leq 1$}, \\
h_\alpha (x) & = \frac{1}{3\log G} \cdot
\begin{cases} \frac{1}{x+G+1} & \mbox{if $x\in \big[ \alpha-2,\frac{1-\alpha}{\alpha}\big)$} \\
\frac{1}{x+1} & \mbox{if $x\in \big[ \frac{1-\alpha}{\alpha}, \frac{\alpha-1}{2-\alpha}\big)$} \\
\frac{1}{x+G-1} & \mbox{if $x\in \big[ \frac{\alpha-1}{2-\alpha}, \alpha \big)$}
\end{cases}
\qquad \mbox{when $1\leq \alpha \leq G$}.
\end{split}
\end{equation*}
\end{cor}

\begin{proof}
The density $h_\alpha=d\nu_\alpha /d\lambda$ is given by
\begin{equation*}
h_\alpha (x) =\frac{1}{3\log G} \int_{R_\alpha (x)} \frac{dy}{(1+xy)^2} ,
\end{equation*}
and a direct calculation shows that it coincides with the expressions above.
\end{proof}

Since the densities $h_\alpha$ are bounded away from 0 and $\infty$, there exists a universal constant
$c>0$ such that
\begin{equation}\label{4.1}
c^{-1} \lambda (A) \leq \nu_\alpha (A) \leq c \lambda (A),\quad \forall A\in \BB_{I_\alpha} .
\end{equation}

For every sequence $\omega=(\omega_k)$ with $\omega_k$ odd integers, denote
$\Delta_n(\omega):=\langle \omega_1,\ldots,\omega_n\rangle_\alpha$
and $q_n(\omega)=q_n(x;\alpha)$, $p_n (\omega)=p_n(x;\alpha)$ if
$x\in \Delta_n (\omega)$.
The map $\varphi_\alpha^n$ maps a non-empty cylinder $\Delta_n(\omega)$ one-to-one onto the interval
$J_n (\omega) :=\varphi_\alpha^n (\Delta_n (\omega))$.
Equations \eqref{3.2} and \eqref{3.4} give that non-empty cylinders correspond to the intervals
\begin{equation}\label{4.2}
\Delta_n (\omega)  =
\left\{ x=\frac{p_n(\omega)+u p_{n-1}(\omega)}{q_n(\omega) +u q_{n-1}(\omega)} :
u\in J_n(\omega) \right\} ,
\end{equation}
where for each $x\in J_n (\omega)$ we have
\begin{equation*}
u=\varphi_\alpha^n (x)=\frac{q_n(\omega)x-p_n(\omega)}{-q_{n-1}(\omega) x+p_{n-1}(\omega)}.
\end{equation*}
Employing \eqref{3.1}, Lemma \ref{Lemma9} and Proposition \ref{Prop8}, we infer
that there exist universal constants $c_3,c_4,c_5 >0$ such that
\begin{equation}\label{4.3}
c_3 \frac{\lambda (J_n(\omega))}{q_n^2(\omega)} \leq \lambda (\Delta_n (\omega)) \leq
c_4 \frac{\lambda (J_n(\omega))}{q_n^2 (\omega)} \leq c_5 C^{-n} .
\end{equation}

Denote by $\GG_n(\alpha)$ the collection of all rank $n$ cylinders
$\langle \omega_1,\ldots,\omega_n \rangle_\alpha$ with $J_n(\omega) =I_\alpha$.
The next lemma is critical in establishing the ergodic properties of $\varphi_\alpha$,
revealing that the collection $\UU(\alpha)$ of
all cylinders in some $\GG_n (\alpha)$ generates the Borel $\sigma$-algebra $\BB_{I_\alpha}$.
We follow literally the proof of Lemma 6 of \cite{Nak}, employing the upper bound on $\lambda (\Delta_n (\omega))$ provided
by \eqref{4.3}.

\begin{lemma}\label{Lemma13}
For almost every $x\in I_\alpha$, there exists a subsequence
$(n_k)$, depending on $x$, such that
\begin{equation*}
\varphi_\alpha^{n_k} \big( \langle \omega_{\alpha,1}(x),\omega_{\alpha,2}(x),\ldots,
\omega_{\alpha,n_k} (x) \rangle_\alpha \big) =I_\alpha,\qquad \forall k\geq 1.
\end{equation*}
\end{lemma}

\begin{proof}
Consider the sets
$C_k :=\{ x\in {\mathbb I}_\alpha:
\varphi_\alpha^k (\langle \omega_{\alpha,1}(x),\ldots,\omega_{\alpha,k}(x)\rangle_\alpha ) =I_\alpha\}$,
$A:=\bigcap_{k=1}^\infty C_k^c$ and $B:=
\bigcup_{n= 1}^\infty \bigcap_{m= n+1}^\infty C_m^c$. We have to prove that $\lambda (B)=0$.

Notice first that upon
\begin{equation*}
\begin{split}
\varphi_\alpha^k (\langle \omega_{\alpha,1}(\varphi_\alpha^n(y)),\ldots ,
\omega_{\alpha,k} (\varphi_\alpha^n (y))\rangle_\alpha ) & =
\varphi_\alpha^k (\langle \omega_{\alpha,n+1} (y),\ldots ,\omega_{\alpha,n+1}(y)\rangle_\alpha ) \\
& =\varphi_\alpha^{n+k} (\langle \omega_{\alpha,1} (y),\ldots ,\omega_{\alpha,n+k}(y)\rangle_\alpha )
\end{split}
\end{equation*}
it follows that $\varphi_\alpha^{-n} (C_k) \subseteq C_{k+n}$, and consequently
$B\subseteq \bigcup_{n=1}^\infty \varphi_\alpha^{-n} (A)$. In conjunction with
\eqref{4.1} and $\nu_\alpha =\nu_\alpha \circ \varphi_\alpha^{-1}$, this shows that it suffices to prove that $\lambda (A)=0$.

For every $n\geq 1$, the set
\begin{equation*}
S_n:=\{ (\omega_1,\ldots,\omega_n)\in \Z^n : \langle \omega_1,\ldots,\omega_n \rangle_\alpha \neq \emptyset ,
\varphi_\alpha^k (\langle \omega_1,\ldots,\omega_k \rangle_\alpha)\neq I_\alpha,
\forall k\in \{ 1,\ldots,n\} \}
\end{equation*}
is contained in $\{ \omega_{\alpha,1}(\alpha),\omega_{\alpha,1}(\alpha-2)\} \times
\{ \omega_{\alpha,2}(\alpha),\omega_{\alpha,2}(\alpha-2)\} \times \cdots \times
\{ \omega_{\alpha,n}(\alpha),\omega_{\alpha,n}(\alpha-2)\}$. Hence $S_n$ contains at
most $2^n$ elements. Estimate \eqref{4.3} then gives that the Lebesgue measure of the
union of cylinders of rank $n$ satisfying $\varphi_\alpha^k (\langle \omega_1,\ldots,\omega_k \rangle_\alpha)\neq I_\alpha$
for every $k\in \{ 1,\ldots,n\}$ is less than or equal to $ c_5 C^{-n} 2^n$. Since $C>2$, we have
$\lim_n \lambda ( \bigcap_{k=1}^n C_k^c ) =0$ and therefore $\lambda (A)=0$.
\end{proof}

Since $\pi_\alpha \Phi_\alpha=\varphi_\alpha \pi_\alpha$ and
$\nu_\alpha (E)=\tilde{\mu}_\alpha (\pi_\alpha^{-1} (E))$, $\forall E\in \BB_{I_\alpha}$,
the dynamical system $(\Omega_\alpha,\BB_{\Omega_\alpha},\tilde{\mu}_\alpha,\Phi_\alpha)$
is an extension of $(I_\alpha,\BB_{I_\alpha}, \nu_\alpha,\varphi_\alpha)$.
Standard arguments as in \cite{Nak} show the minimality of this extension, in the sense that
$\bigvee_{n=0}^\infty \Phi_\alpha^n \pi_\alpha^{-1} (\BB_{I_\alpha}) =\BB_{\Omega_\alpha}$, hence
we have

\begin{theorem}\label{Thm14}
$(\Omega_\alpha,\BB_{\Omega_\alpha},\tilde{\mu}_\alpha,\Phi_\alpha)$ gives the
natural extension of $(I_\alpha,\BB_{I_\alpha},\nu_\alpha,\varphi_\alpha)$
in the sense of \cite{Ro}.
\end{theorem}

For any interval $[x,y] \subseteq J_n(\omega)$,
the set $\varphi_\alpha^{-n} ([x,y]) \cap \Delta_n (\omega)$ is an interval.
Employing Proposition \ref{Prop8}, Lemma \ref{Lemma9} and \eqref{4.2},
there exist universal constants $c_6,c_7 >0$ such that
\begin{equation*}
c_6 \frac{y-x}{q_n^2(\omega)} \leq \lambda (\varphi_\alpha^{-n}([x,y]) \cap \Delta_n(\omega)) =
\frac{y-x}{(q_n(\omega)+xq_{n-1}(\omega))(q_n(\omega) +yq_{n-1}(\omega))} \leq c_7 \frac{y-x}{q_n^2(\omega)}.
\end{equation*}
Upon \eqref{4.3}, there now exist
universal constants $c_8,c_9>0$ such that, for every rank $n$ cylinder $\Delta_n(\omega) \in \GG_n (\alpha)$
and $[x,y] \subseteq I_\alpha$, we have
\begin{equation*}
c_8 (y-x) \lambda (\Delta_n(\omega)) \leq
\lambda (\varphi_\alpha^{-n}([x,y]) \cap \Delta_n(\omega)) \leq
c_9 (y-x) \lambda (\Delta_n(\omega)) .
\end{equation*}
This further gives, whenever $\Delta_n(\omega) \in \GG_n (\alpha)$ and $A\in \BB_{I_\alpha}$,
\begin{equation}\label{4.4}
c_8 \lambda (A) \lambda (\Delta_n(\omega)) \leq
\lambda (\varphi_\alpha^{-n}(A) \cap \Delta_n(\omega)) \leq
c_9 \lambda (A) \lambda (\Delta_n(\omega)) .
\end{equation}

\begin{proposition}\label{Prop15}
The measure-preserving transformation $(I_\alpha,\BB_{I_\alpha},\varphi_\alpha,\nu_\alpha)$ is ergodic.
\end{proposition}

\begin{proof}
Let $A\in \BB_{I_\alpha}$ such that $\varphi_\alpha^{-1}(A)=A$. By \eqref{4.4} we have
\begin{equation*}
\lambda (A \cap \Delta_n(\omega)) \geq c_8 \lambda (A) \lambda (\Delta_n(\omega)),\quad
\forall n\geq 1,\ \forall \Delta_n (\omega) \in \GG_n(\alpha).
\end{equation*}
Since $\UU(\alpha)$ generates $\BB_{I_\alpha}$, this further yields
\begin{equation*}
\lambda (A \cap B) \geq c_8 \lambda (A) \lambda (B),\quad
\forall B \in \BB_{I_\alpha}.
\end{equation*}
Choosing $B=A^c$, it follows that either $\lambda (A)=0$ or $\lambda (A^c)=0$,
showing that $\varphi_\alpha$ is ergodic.
\end{proof}

The endomorphism $\varphi_\alpha$ has stronger ergodic
properties. In particular, it turns out to be also exact
in the sense of Rohlin \cite{Ro}, which means that the tail $\sigma$-algebra
$\bigcap_{n=0}^\infty \varphi_\alpha^{-n} \BB_{I_\alpha}$ only consists of null or co-null sets.
A convenient equivalent  formulation is that for any positive measure set $A\in \BB_{I_\alpha}$,
 $\lim_n \nu_\alpha (\varphi_\alpha^n(A))=1$.

\begin{theorem}\label{Thm16}
The measure-preserving transformation $(I_\alpha,\BB_{I_\alpha},\varphi_\alpha,\nu_\alpha)$ is exact,
and in particular it is mixing of all orders (\cite{Ro}).
\end{theorem}

\begin{proof} We use the exactness criterion proved in \cite[Thm. 4.2]{Ro}.
In our situation this amounts to showing that there exists a constant $q>0$ such that
the following inequality holds for every $n\geq 1$, every cylinder $\Delta_n(\omega) \in \GG_n(\omega)$,
and every Borel set $X \subset \Delta_n (\omega)$:
\begin{equation}\label{4.5}
\nu_\alpha (\varphi_\alpha^n (X)) \leq q \frac{\nu_\alpha (X)}{\nu_\alpha (\Delta_n(\omega))}.
\end{equation}

Taking $E:=\varphi_\alpha^n (X)\in \BB_{I_\alpha}$, estimate \eqref{4.4} provides
\begin{equation*}
\lambda (X) =\lambda ( \varphi_\alpha^{-n} (E) \cap \Delta_n (\omega)) \geq c_8 \lambda
(\varphi_\alpha^n (X)) \lambda ( \Delta_n(\omega)) .
\end{equation*}

Upon \eqref{4.1}, this yields inequality \eqref{4.5} with $q=c^3 c_8$.
\end{proof}

\begin{cor}\label{Cor17} For every set $A \in \BB_{I_\alpha}$, we have
\begin{equation*}
\lim\limits_n \lambda (\varphi_\alpha^{-n} (A)) = 2\nu_\alpha (A)\end{equation*}
\end{cor}

\begin{proof}
The mixing property of $\varphi_\alpha$,
\begin{equation*}
\int_{I_\alpha} (f \circ \varphi_\alpha^n) g\, d\nu_\alpha \ \stackrel{n}{\longrightarrow}\
\int_{I_\alpha} f\, d\nu_\alpha \ \int_{I_\alpha} g\, d\nu_\alpha,\qquad\forall f,g\in L^2 (\nu_\alpha),
\end{equation*}
shows in particular that
\begin{equation*}
\lambda (\varphi_\alpha^{-n} (A)) =
\int_{I_\alpha} (\chi_A \circ \varphi_\alpha^n )\,  \frac{d\nu_\alpha}{h_\alpha} \
\stackrel{n}{\longrightarrow} \  \int_{I_\alpha} \chi_A d\nu_\alpha \
 \int_{I_\alpha} \frac{d\nu_\alpha}{h_\alpha}  =2\nu_\alpha (A).
\qedhere
\end{equation*}
\end{proof}

Exactness of $\varphi_\alpha$ also implies that its natural
extension $\Phi_\alpha$ is a $K$-automorphism (cf., e.g., \cite[Prop. 3.4]{Ro} or \cite[Thm.3, p. 289]{CFS}),
and that $\nu_\alpha$ is the unique $\sigma$-finite Lebesgue absolutely continuous $\varphi_\alpha$-invariant measure
(cf., e.g., \cite[Thm. 12.1.3]{Sch}).

\begin{lemma}\label{L18}
For every $\alpha\in [g,G]$ we have
\begin{equation*}
J(\alpha):=\int_{\alpha-2}^\alpha \log \vert x\vert \, h_\alpha (x)\, dx =-\frac{\pi^2}{18\log G} .
\end{equation*}
\end{lemma}

\begin{proof}
When $g\leq \alpha \leq 1$ we have
\begin{equation*}
\begin{split}
J(\alpha) = \int_{\alpha-2}^0 & \frac{\log(-x)\, dx}{x+G+1}
+\int_0^{\frac{1}{\alpha}-1} \frac{\log x\, dx}{x+G+1}
+ \int_{\frac{\alpha-1}{2-\alpha}}^0 \frac{\log (-x)\, dx}{x+G-1}  \\
& +\int_0^\alpha \frac{\log x\, dx}{x+G-1}
- \int_{\frac{\alpha-1}{2-\alpha}}^0 \frac{\log (-x)\, dx}{x+1}
- \int_0^{\frac{1}{\alpha}-1} \frac{\log x\, dx}{x+1} .
\end{split}
\end{equation*}
Applying the Fundamental Theorem of Calculus and equalities
$\frac{1}{\alpha}+G =\frac{\alpha+G-1}{(G-1)\alpha}$ and
$\frac{\alpha-1}{2-\alpha} +G-1 =\frac{\alpha+G-1}{(G+1)(2-\alpha)}$, we find
\begin{equation*}
\begin{split}
(\alpha+G-1)J^\prime (\alpha)  = & -\log (2-\alpha) -\frac{G-1}{\alpha} \log \left(\frac{1-\alpha}{\alpha}\right)
-\frac{G+1}{2-\alpha} \log \left( \frac{1-\alpha}{2-\alpha}\right) \\
& + \log\alpha +\frac{\alpha+G-1}{2-\alpha} \log \left( \frac{1-\alpha}{2-\alpha}\right)
+\frac{\alpha+G-1}{\alpha} \log \left( \frac{1-\alpha}{\alpha}\right) = 0.
\end{split}
\end{equation*}

When $1\leq \alpha \leq G$, a similar computation provides $J^\prime (\alpha)=0$ as well.

Employing some basic properties of dilogarithms it was shown in \cite{Rie}
that $J(G) =-\frac{\pi^2}{18\log G}$ and in \cite{Sch0} that
$J(1)=-\frac{\pi^2}{18\log G}$.
\end{proof}

The following are standard consequences of ergodicity of $\varphi_\alpha$:

\begin{cor}\label{Cor19}
For every $\alpha \in [g,G]$ and almost every $x\in \II_\alpha$:
\begin{itemize}
\item[(i)]
$\displaystyle \lim\limits_n \frac{1}{n} \log | q_n (x;\alpha) x-p_n (x;\alpha)| =
- \frac{\pi^2}{18 \log G}$.
\item[(ii)]
$\displaystyle \lim\limits_n \frac{1}{n} \log q_n(x;\alpha) = \frac{\pi^2}{18 \log G}$.
\item[(iii)]
$\displaystyle \lim\limits_n \frac{1}{n} \log \left| x-\frac{p_n(x;\alpha)}{q_n(x;\alpha)} \right|
=-\frac{\pi^2}{9\log G}$.
\end{itemize}
\end{cor}

\begin{proof}
(i) Equation \eqref{3.3} yields
\begin{equation}\label{4.6}
\prod\limits_{k=0}^n | \varphi_\alpha^k (x)| =
\frac{| q_n (x;\alpha) x-p_n (x;\alpha)|}{| x|} ,
\quad \forall x\in \II_\alpha .
\end{equation}
Since $\varphi_\alpha$ is ergodic and $\log | x| \in L^1 (\nu_\alpha)=L^1 (\lambda )$,
the ergodic theorem yields
\begin{equation}\label{4.7}
\lim\limits_n \frac{1}{n+1} \sum\limits_{k=0}^n \log | \varphi_\alpha^k (x)|
=\int_{I_\alpha} \log | x| \, d\nu_\alpha (x) =-\frac{\pi^2}{18 \log G} .
\end{equation}
The statement follows from \eqref{4.6} and \eqref{4.7}.

(ii) follows from (i) and $\frac{1}{q_{n+1}} \leq | q_n (x;\alpha) x-p_n (x;\alpha)|
\leq \frac{c_1}{q_n}$, which was proved in Proposition \ref{Prop10}.
\par (iii) is a consequence of (i) and (ii).
\end{proof}

\begin{theorem}\label{Thm20}
For every $\alpha \in [g,G]$ the entropy of $(\varphi_\alpha,\nu_\alpha)$ and
$(\Phi_\alpha,\mu_\alpha)$ is
\begin{equation*}
h_{\mu_\alpha}(\Phi_\alpha)= h_{\nu_\alpha}(\varphi_\alpha) = \frac{\pi^2}{9\log G} .
\end{equation*}
\end{theorem}

\begin{proof}
Following the argument in the proof of Theorem 3 in \cite{Nak}, this can be deduced from
the Shannon-McMillan-Breiman theorem, extended by Chung for countable partitions, which gives
\begin{equation*}
h_{\nu_\alpha} (\varphi_\alpha) =-\lim\limits_{n\rightarrow \infty} \frac{1}{n} \log
\nu_\alpha (\langle \omega_{\alpha,1} (x),\ldots,\omega_{\alpha,n_k}(\alpha)\rangle_\alpha )
\quad \mbox{\rm for a.e. $x\in I_\alpha$.}
\end{equation*}
Choose $x$ and some subsequence $n_k=n_k(x)$ as in Lemma \ref{Lemma13}. Then
by Proposition \ref{Prop10} we see that
\begin{equation*}
\frac{c_3}{q_{n_k}^2 (x;\alpha)} \leq \lambda (\langle \omega_{\alpha,1} (x),\ldots ,\omega_{\alpha,n_k} (x)
\rangle_\alpha ) \leq \frac{c_4}{q_{n_k}^2 (x;\alpha)} .
\end{equation*}

Taking also into account \eqref{4.1} we find
\begin{equation*}
h_{\nu_\alpha} (\varphi_\alpha) =2\lim\limits_{k\rightarrow \infty}
\log q_{n_k} (x;\alpha) =\frac{\pi^2}{9\log G} .
\end{equation*}
Entropy being preserved under the natural extension, the equality
$h_{\mu_\alpha}(\Phi_\alpha)=h_{\nu_\alpha} (\varphi_\alpha)$ follows.
\end{proof}

\begin{center}
\begin{figure}
\includegraphics[scale=1.2, bb = 0 0 250 200]{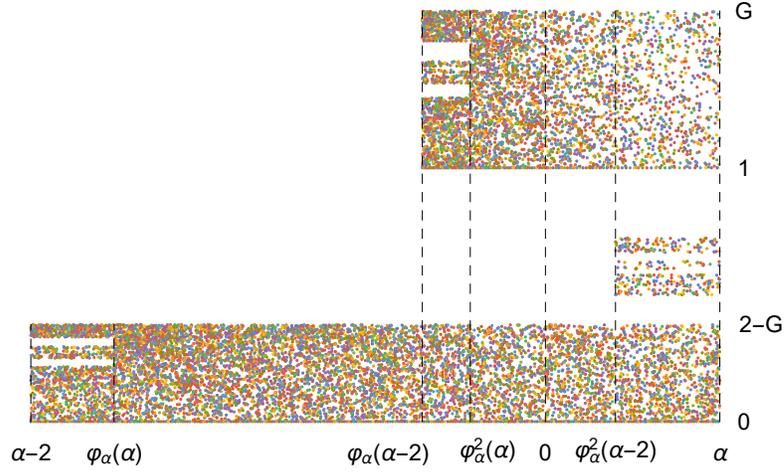}
\caption{\small The set $\{ \Phi_\alpha^n (I_\alpha \times \{ 0\}):0\leq n\leq 50\} $ with
$\alpha =0.9 g$}\label{Figure2}
\end{figure}
\end{center}

When $\alpha <g$, numerical experiments (see especially Fig. \ref{Figure2}) suggest that
the shape of the natural extension domain
$\Omega_\alpha =\overline{\{ \Phi_\alpha^n (I_\alpha \times \{ 0\}): n\geq 0\}}$
becomes considerably more intricate, as it happens in the situation of
the $\alpha$-RCF expansion with $\alpha <\frac{1}{2}$. This will be
investigated elsewhere.

\section*{Acknowledgments}
We are grateful to Niels Langeveld for pointing out two inaccuracies in an
earlier draft of the paper, to Pierre Arnoux for useful comments, and to the referee for
careful reading and a number of useful comments and corrections that improved
the paper, especially Section 4.

\end{document}